\documentclass[12pt, a4paper, reqno]{amsart}
\usepackage[english]{babel}
\usepackage{amsmath,amsfonts,amssymb,amsthm}
\usepackage{hyperref}
\usepackage{graphicx}
\usepackage{color}
\usepackage{caption}
\usepackage{subcaption}
\usepackage{longtable}

\usepackage[margin=2.5cm]{geometry}

\newtheorem{thm}{Theorem}[section]
\newtheorem{lem}[thm]{Lemma}
\newtheorem{cor}[thm]{Corollary}

\theoremstyle{definition}
\newtheorem{defi}[thm]{Definition}
\newtheorem{exam}[thm]{Example}
\newtheorem{rem}[thm]{Remark}

\newcommand { \ib }[1] {\textit{\textbf{#1}}}
\newcommand{\bd}{\stackrel{\rm bd}{\sim}}

\newcommand { \supp }{\mathop{\rm{supp}}\nolimits}

\DeclareMathOperator{\dens}{dens}

\newcommand{\R}{\mathbb{R}}
\newcommand{\Q}{\mathbb{Q}}
\newcommand{\Z}{\mathbb{Z}}
\newcommand{\N}{\mathbb{N}}
\newcommand{\X}{\mathbb{X}}
\newcommand{\Hb}{\mathbb{Hb}}

\begin{document}
\renewcommand{\ib}{\mathbf}
\renewcommand{\proofname}{Proof}
\renewcommand{\phi}{\varphi}
\newcommand{\conv}{\mathrm{conv}}
%\makeatletter \headsep 10 mm \footskip 10 mm
%\renewcommand{\@evenhead}%


\title[]{Weighted $1\times1$ cut-and-project sets in bounded
distance to a lattice}
\author{Dirk Frettl\"oh}
\address{Technische Fakult\"at, Bielefeld University}
\email{dfrettloeh@techfak.uni-bielefeld.de}

\author{Alexey~Garber}
\address{School of Mathematical \& Statistical Sciences, The University of Texas Rio Grande Valley, 1 West University Blvd., Brownsville, TX 78520, USA.}
\email{alexeygarber@gmail.com}

%\thanks{{$^{*}$}Partially supported by a grant of the 
%Dynasty Foundation and by the grant ``Leading Scientific Schools'' NSh-4833.2014.1.}

\date{\today}

\begin{abstract}
Recent results of Grepstad and Larcher are used to show that 
weighted cut-and-project sets with one-dimensional physical space 
and one-dimensional internal space are bounded distance equivalent to 
some lattice if the weight function $h$ is continuous on the internal space,
and if $h$ is either piecewise linear, or twice differentiable with bounded
curvature. 
\end{abstract}

\maketitle

\section{Introduction}

A \emph{Delone set} is a non-empty set $\Lambda$ of points in some metric space 
$\X$ such that (1) there is $r>0$ such that each open ball of radius $r$ centered at a point $x$ of $\Lambda$
contains no other points of $\Lambda$, and (2) there is $R>0$ such 
that each closed ball of radius $R$ centered at a point $x$ of $\Lambda$ contains at least one more point of $\Lambda$. 
Depending on the context, Delone sets are also called separated nets,
or $(r,R)$-sets. Two Delone sets $\Lambda, \Lambda'$ in the same metric space are called 
\emph{bounded distance equivalent} ($\Lambda \bd \Lambda'$) if there 
is a bijection $\phi \colon  \Lambda \to \Lambda'$ such that $|x - \phi(x)|$ 
is uniformly bounded.
In 1993 M.~Gromov asked whether any Delone set $\Lambda$ in $\R^2$
is \emph{bilipschitzequivalent} with $\Z^2$ \cite{G}; i.e., whether
there is a bijection from $\Lambda$ to $\Z^2$ such that the bijection is Lipschitz continuous 
in both directions. In 1998 D.~Burago and B.~Kleiner, and independently 
C.~McMullen, gave a negative answer \cite{BK,McM}. 
%This inspired the search for particuler classes of Delone sets that \emph{are}
%bilipschitzequivalent to $\Z^2$. 
The analogous question for the hyperbolic plane $\Hb^2$ was
answered positively by Bogopolskii \cite{Bog} by showing that all 
Delone sets in $\Hb^2$ are bounded distance equivalent to each other.
Bounded distance equivalence implies bilipschitzequivalence.

Even before that physicists asked whether some given crystallographic
or quasicrystallographic Delone set $\Lambda$ in $\R^2$ or $\R^3$ 
has an ``average lattice'' of the form $a \Z^2$; i.e. whether
there is $a>0$ such that $\Lambda \bd a \Z^2$. A \emph{lattice} 
in $\R^d$ is the $\Z$-span $\langle v_1, \ldots, v_d \rangle_{\Z}$ 
of $d$ linearly independent vectors $v_i \in \R^d$. In 
\cite{DO1} it is shown that any two lattices in $\R^d$ 
with equal density are bounded distance equivalent. 
In \cite{DO2} a sufficient condition for a cut-and-project set (CPS)
being bounded distance equivalent to some lattice with the same
density is given. For a definition of a CPS
see below. There is no precise mathematical definition of
a quasicrystal; but often when speaking of a (mathematical) 
quasicrystal a CPS set is meant. 

Recently bounded distance equivalence of Delone sets 
did get some attention, see e.g. 
\cite{Lac, GLev, H, HK, HKK} and references therein.
A frequently exploited connection is the correspondence
between (certain) CPS and (certain) bounded remainder sets 
for (discrete) toral rotations. Given a set $S \subset [0,1)$ and
some (irrational) $\alpha>0$ the \emph{deficiency} (or 
discrepancy) of $S$ with respect to some $x \in \R$ is
\[ D_n(S,x) := \sum_{k=0}^{n-1} 1_S(x+k \alpha \bmod 1) - n \lambda(S), \]
where $\lambda$ denotes Lebesgue measure on $\R$.  
A set $S \subset [0,1)$ is called a \emph{bounded remainder set} 
(BRS) with respect to $\alpha$ if there is $C>0$ such that for 
almost all $x$ we have $\sup\limits_{n\in\N} |D_n(S,x)|<C$. As we will see, 
for our purposes the $x$ plays no role; it is included in the 
definition only because in some contexts there is an exceptional
null-set of $x$ to consider.

A profound theorem of Kesten \cite{Kes} shows 
that an interval $[a,b] \subset [0,1)$ is a BRS
for the discrete toral rotation $n \alpha \bmod 1$ on the one-dimensional torus 
if and only if $b-a \in \Z+\alpha \Z$. Applied to CPS this 
proves for instance that the Fibonacci sequence, defined by
a CPS with lattice $\langle (1,1)^T, (\frac{1+\sqrt{5}}{2},
\frac{1-\sqrt{5}}{2})^T \rangle_{\Z}$ and window 
$[0,\frac{1+\sqrt{5}}{2})$ is bounded distance
equivalent to some lattice, whereas the Half-Fibonacci
sequence using the same lattice but window $[0,\frac{1+\sqrt{5}}{4})$,
is not bounded distance equivalent to any lattice. See Examples \ref{ex:fib1} and \ref{ex:fib2} for more details.

In this paper we exploit the connection between continuous toral 
rotations and weighted cut-and-project sets. Our main result
Theorem \ref{thm:diraccomb} uses two theorems
of \cite{GL} on continuous toral rotations. It shows that
many weighted $1 \times 1$ CPS where the window is an interval
and the weight function $h$ is continuous and supported on $W$
(hence $h$ equals 0 at the endpoints of the interval)
are bounded distance equivalence to some lattice, with no
restrictions on the length of the window. This is in strong contrast
with the discrete case, see Kesten's theorem mentioned above, 
respectively the Half-Fibonacci example.

\textbf{Notation:} Throughout the paper, $\lambda$ denotes 
$d$-dimensional Lebesgue measure (where $d=1$ or $d=2$, depending 
on the context). The Dirac measure in $x$ is denoted $\delta_x$. 

\section{Cut-and-project sets} \label{sec:cps}

A \emph{cut-and-project set} (CPS, aka \emph{model set}) 
$\Lambda$ is given by a collection of maps and spaces:
\[ \begin{array}{ccccc}
G & \stackrel{\pi_1}{\longleftarrow} & G \times H & 
\stackrel{\pi_2}{\longrightarrow} & H\\
\cup & & \cup & & \cup\\
\Lambda & & \Gamma & & W
\end{array} \]
where in general $G$ and $H$ are locally compact abelian groups.
Furthermore, $\Gamma$ is a lattice (i.e., a discrete cocompact subgroup)
in $G \times H$, $W$ is a relatively compact 
set in $H$, and $\pi_1$ and $\pi_2$ are projections to $G$ and
to $H$ respectively, such that 
$\pi_1|_\Gamma$ is one-to-one, and $\pi_2(\Gamma)$ is
dense in $W$. Then 
\[ \Lambda = \{ \pi_1(x) \, | \, x \in \Gamma, \, \pi_2(x) \in W \} \]
is called a CPS. 

Throughout this paper we will always have $G=\R$ and $H=\R$,
hence we call the resulting CPS sometimes $1 \times 1$-CPS
in order to distinguish them from CPS where $G$ or $H$ have higher 
dimension.
For the sake of clarity, we will refer to these spaces as $G$ and $H$ (rather than 
$\R$ and $\R$) in order to distinguish the space $G$ 
supporting the CPS $\Lambda$ (often called \emph{direct space})
from the space $H$ supporting $W$ (often called \emph{internal space}).

It does not really matter
whether $\Gamma$ is a proper lattice, or a translate of some lattice,
since translating the lattice by $z$ yields the same CPS (shifted 
by $\pi_1(z)$) as translating the window $W$ by $\pi_2(z)$.
In general, translating the window corresponds 
just to changing the CPS $\Lambda$ to another CPS $\Lambda'$ that is 
locally indistinguishable from $\Lambda$ provided $\pi_2(\Gamma)$ has empty intersection with boundaries of windows for $\Lambda$ and $\Lambda'$, in the sense that a copy
of each local piece of $\Lambda$ appears in $\Lambda'$, and vice versa.

The \emph{density} of a CPS is the average number of points
per unit area. It is known that the density of a CPS exists
and equals
\begin{equation} \label{eq:dens-cps}
\dens{\Lambda} =  \frac{\lambda(W)}{|\det(M_{\Gamma})|}, 
\end{equation}
where $M_{\Gamma}$ is the matrix whose columns are the spanning 
vectors of the lattice $\Gamma$. See \cite[Thm.~7.2]{BG} and
references there for details.

\begin{exam} \label{ex:fib1}
%The (symbolic) Fibonacci sequence can be generated by applying
%the map $\sigma \colon  \, a \mapsto ab$, $b \mapsto a$ repeatedly to 
%the letter pair $a|a$: $\sigma(a|a)=ab|ab$, $\sigma^2(a)=aba|aba$, 
%$\sigma^4(a)=abaababa|abaababa$, $\sigma^6(a)=abaababaabaababaababa|
%abaababaabaababaababa$, $\ldots$. This symbolic sequence can be 
%transformed into a Delone set in $\R$ by assigning an interval of length 
%$\tau = \frac{\sqrt{5}+1}{2}$ to $a$ and an interval of length 1 to $b$. 
%The corresponding Delone set $\Lambda$ then consists of the endpoints of the intervals.
%This Delone set can be defined via a CPS, too.
%
Probably one of the most prominent CPS is the Fibonacci sequence. The corresponding CPS has $G=\langle (1,0)^T \rangle_{\R}$, 
$H\langle (0,1)^T \rangle_{\R}$, $W=[-\frac{1}{\tau}, 1[ \subset H$, lattice 
$\Gamma = \langle \big( \begin{smallmatrix} 1\\ 1 \end{smallmatrix} 
\big), \big( \begin{smallmatrix} \tau \\ -\tau^{-1} 
\end{smallmatrix} \big) \rangle_{\Z}$, and $\pi_1$ and $\pi_2$ are 
orthogonal projections to $G$, respectively to $H$.

See also Example \ref{ex:fib2} below.
\end{exam}

Weighted CPS are a generalisation of the notion of a CPS. A 
\emph{weighted} CPS is a Dirac comb $\sum\limits_{x \in \Lambda}
h(x^{\star}) \delta_x$, where $h\colon H\to \R$, $h|_{H\setminus W}=0$, and the restriction $h|_W$ is continuous, and 
$x^{\star} := \pi_2(\pi^{-1}_1(x))$. Here, $\pi^{-1}_1(x)$ makes
sense since $\pi_1|_{\Gamma}$ is one-to-one. A weighted CPS with
constant weight function $h(x)=1$ for all $x \in W$(and $h(x)=0$
for $x \notin W$) is just an ordinary CPS, viewed as a measure. 
Weighted Dirac combs and weighted CPS are relevant in the study
of diffraction properties of CPS, see \cite{BG} and references
therein. It is easy to see that the density formula \eqref{eq:dens-cps}
for CPS generalises to weighted CPS as follows:
\begin{equation} \label{eq:dens-weighted}
\dens{\Lambda} =  \frac{\int_{W} h(t) dt}{|\det(M_{\Gamma})|}.
\end{equation}

\section{BRS for continuous rotations and weighted CPS}

In order to utilize the results of \cite{GL} we generalize the notion
of bounded distance equivalence from point sets to measures.

\begin{defi}\label{def:bdmes}
Two measures $\mu, \nu$ on $\R$ are \emph{bounded distance 
equivalent}, if there is $C>0$ such that for all $a,b \in \R$ with $a<b$ 
\[ |\mu([a,b])-\nu([a,b])| < C. \]
\end{defi}

\begin{rem}
The only restriction we impose on the measures $\mu$ and $\nu$ in the definition above is that all intervals (open, closed, semi-open) are measurable with respect to $\mu$ and $\nu$. However, in the further discussion we will mostly work with multiples of standard Lebesgue measure and with (weighted) Dirac comb measures defined for discrete sets, see the definition below.

It is easy to see that the relation above defines an equivalence relation.
\end{rem}

Since a point set $\Lambda$ in $\R$ can be identified with a measure
$\sum\limits_{x \in \Lambda} \delta_x$ it is not hard to see that Definition \ref{def:bdmes} reduces for Delone sets to the definition of bounded 
distance equivalence above. Nevertheless, we spell out the details 
in the proof of the next lemma. 
\begin{lem}
Two Delone sets $\Lambda, \Lambda'$ in $\R$ are bounded distance
equivalent as point sets if and only if the corresponding Dirac
combs $\omega = \sum\limits_{x \in \Lambda} \delta_x$ and 
$\omega' = \sum\limits_{x' \in \Lambda'} % Rev 2, comment 7
\delta_{x'}$ are bounded distance equivalent as measures.
\end{lem}
\begin{proof}
Without loss of generality let $\Lambda = \{ \ldots, x_{-1}, x_0=0, x_1, 
\ldots \}$ (with $x_i<x_j$ if $i<j$) and $\Lambda' = \{ \ldots, x'_{-1}, 
x'_0=0, x'_1, \ldots \}$ (with $x'_i<x'_j$ if $i<j$). Let $\omega$ 
respectively $\omega'$ be the corresponding Dirac combs. Let $r>0$ be such a number that if $i\neq j$, then $|x_i-x_j|\geq r$ and $|x_i'-x_j'|\geq r$. The constant $r$ can be taken as the smallest of two smaller radii in the Delone property of $\Lambda$ and $\Lambda'$.

If there is a bounded distance bijection between $\Lambda$ 
and $\Lambda'$ then $x_i \mapsto x'_i$ is a bounded distance bijection, 
too. Hence there is $C'>0$ such that $|x_i - x'_i|<C'$ for all $i$. 

Let $x'_{i+\ell}$ be the largest $x' \in \Lambda'$ with $x'<x_i$. By 
the Delone property the interval $[x'_i,x_i]$ contains at most 
$\frac{|x_i - x'_i|}{r}+1$ points of $\Lambda'$, hence
\[ |\ell| \le \frac{|x_i - x'_i|}{r}+1 < \frac{C'}{r}.\] 
Thus the difference 
\[ |\omega([a,b]) - \omega'([a,b])| = | \sum\limits_{x \in \Lambda \cap [a,b]} 
\delta_x([a,b]) -  \sum\limits_{x' \in \Lambda' \cap [a,b]} \delta_{x'}([a,b]) | \]
is bounded by the number of points $x_i \in [a,b]$ such that $x'_i \notin
[a,b]$ (or vice versa). Thus 
\[ |\omega([a,b]) - \omega'([a,b])| < 2 \frac{C'}{r},  \]
where $C'$ and $r$ depend only on $\Lambda$ and $\Lambda'$. 

Conversely, if $|\omega([-n,n]) - \omega'([-n,n])| < C$ for all $n$,
then the number of points in $\Lambda \cap [-n,n]$ deviates at most 
by $C$ from the number of points in $\Lambda' \cap [-n,n]$. For $i\geq 0$, if 
$x_i \in [-n,n]$ but $x'_i \notin [-n,n]$, then $[x_i, x'_i[$ can
contain at most $\frac{|x'_i-x_i|}{r}$ points of $\Lambda'$; again 
by the Delone property of $\Lambda'$. Hence 
\[ \frac{|x'_i - x_i|}{r} < C \mbox{ respectively } |x'_i-x_i| < Cr. \] 
where $C$ and $r$ depend only on $\Lambda$ and $\Lambda'$. The same holds
for $x_i,x'_i$ with $i<0$. 
\end{proof}

The paper \cite{GL} studies BRSs of the continuous 
analogue of the discrete toral rotations above. We state two definitions
from \cite{GL}, slightly simplified for our purposes. 
\begin{defi}
 Let $x = (x_1, x_2) \in [0,1]^2$, and let $\alpha \in \R \setminus \Q$. 
 We say that the function $X  \colon  [0, \infty) \mapsto [0,1]^2$ defined by
 \begin{equation*}
  X(t) = \left( x_1+t \bmod 1, x_2+\alpha t \bmod 1 \right)
 \end{equation*}
is the two-dimensional continuous irrational rotation with slope $\alpha$ 
and starting point $x$.
\end{defi}

The notion of deficiency translates as follows.
\begin{defi}
\label{def:cbrs}
Let $P \subset [0,1]^2$ be an arbitrary measurable set with 
Lebesgue measure $\lambda (P)$. We say that $P$ is a \emph{bounded 
remainder set} (BRS) for the continuous irrational rotation with slope $\alpha>0$ 
and starting point $x=(x_1, x_2) \in [0,1]^2$ if the distributional error
\begin{equation}
 \label{eq:discrepancy}
 \Delta_t(P, \alpha, x) = \int_0^t 1_P \left( x_1+s \bmod 1, x_2+\alpha s 
\bmod 1 \right) \, ds  - t \lambda (P)
\end{equation}
is uniformly bounded for all $t > 0$. Here, $1_P$ denotes the characteristic 
function for the set $P$.
\end{defi}

The following simple observation will be useful in the sequel. It can be
shown easily by spelling out the definition (resp., definitions, since it 
holds in both cases, discrete toral rotations and continuous toral rotations). 
\begin{lem} \label{lem:sum-diff-brs}
Let $P, P'$ be BRSs. If $P \cap P' = \varnothing$ then the union 
$P \cup P'$ is a BRS, too. If $P'\subset P$ then the difference 
$P \setminus P'$ is a BRS, too.
\end{lem}

Two of the main results in \cite{GL} are the following.
\begin{thm}
 \label{thm:mainpoly}
 For almost all $\alpha>0$ and every $x \in [0,1]^2$, every polygon 
 $P \subset [0,1]^2$ with no edge of slope $\alpha$ is a BRS for the 
 continuous irrational rotation with slope $\alpha$ and starting point $x$.
\end{thm}
\begin{thm}
 \label{thm:mainconvex}
 For almost all $\alpha>0$ and every $x \in [0,1]^2$, every convex set 
 $P \subset [0,1]^2$ whose boundary $\partial P$ is a twice continuously 
 differentiable (regular) curve with positive curvature at every point is a BRS for the 
 continuous irrational rotation with slope $\alpha$ and starting point $x$. 
\end{thm}

%%%% Rev 2, comm 8
\begin{rem}\label{rem:curvature}
From a geometric point of view the curvature $\kappa(x)$ at $x\in \gamma$ can be defined for a twice continuously differentiable regular curve $\gamma$ as the reciprocal of the radius of a circle (or a line, in that case $\kappa(x)=0$) that gives the best approximation of $\gamma$ at $x$. Here regularity means that there exists a parametrization $\gamma = r(t)$ of the curve, for example a natural parametrization with its length, such that $\dot{r}$ is never a zero vector. Here dot denotes the derivative with respect to the variable $t$. We will also assume only regular parametrizations in the sequel.

If $\gamma = r(t)$ is a parametrization of a regular curve, then $\kappa(t)=\frac{|\dot{r}\times \ddot{r}|}{|\dot{r}|^3}$ where the numerator is the length of the cross-product in the ambient $3$-space. If $\gamma = r(s)$ is parametrized with its length, then $\kappa(s)=|r''(s)|$. If the curve is given by equation $y=f(x)$ in standard rectangular coordinate system, then $x$ can be treated as a parameter and $$\kappa(x)=\frac{|f''(x)|}{(1+(f'(x))^2)^{3/2}}.$$

Note, that the value of the curvature at given point $x\in \gamma$ does not depend on a (regular) parametrization of $\gamma$ in a neighborhood $x$ because the geometric description of the curvature given above.

We refer to \cite{Pog}, or almost any other differential geometry textbook, for more details about geometry of planar curves.
\end{rem}

To a BRS $P$ and an irrational slope $\alpha$
as above one can associate a weighted CPS
as follows; see also Figure \ref{fig:cps-weighted}. 
The direct space is $G=\big( \begin{smallmatrix} 1\\ \alpha \end{smallmatrix}
\big) \R$, the internal space is
the orthogonal complement $H= \big( \begin{smallmatrix} 1\\ \alpha 
\end{smallmatrix} \big)^{\perp}$ of $G$ in $\R^2$. The projections
$\pi_1$ and $\pi_2$ are the orthogonal projections to $G$, 
respectively to $H$, and $W=\pi_2(P)$.  
Since $P$ is connected, $W$ is a line segment in $H$, so we have 
$W = [h_1,h_2]$ for some $h_i \in H$. 
Because of the properties of $P$ (either positive curvature, or no
slope in direction $\alpha$) there is exactly one point $z \in [0,1]^2$
such that $\pi_2(z)=h_1$. Let $\Gamma \subset G \times H$
be $z+ \Z^2$. Hence $\Gamma$ is not actually a lattice here,
but a translation of the lattice $\Z^2$. This makes no difference,
see the remark in the definition of a CPS in Section \ref{sec:cps}. 
Since $\alpha$ is 
irrational, $\pi_1|_{\Gamma}$ is one-to-one, and $\pi_2(\Gamma)$
is dense in $W$. Let $\Lambda$ be the CPS defined by these data. 

\begin{figure}
\includegraphics[width=.9\textwidth]{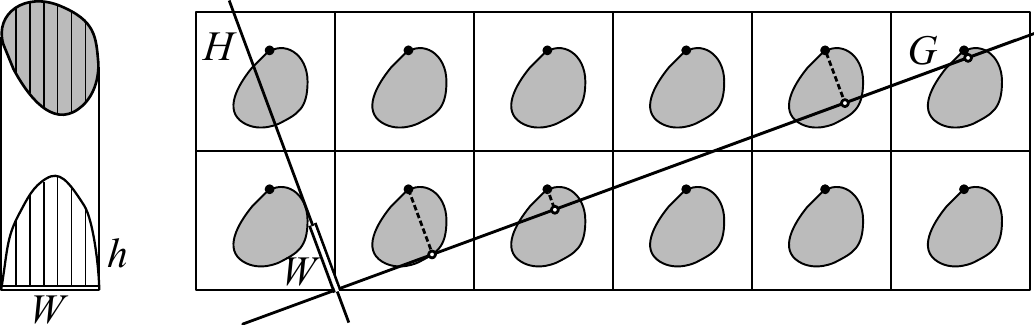}
\caption{A CPS tailored to BRS for continuous toral rotations.
The direct space $G$ is the line 
% $\big( \begin{smallmatrix} 1\\ \alpha \end{smallmatrix} \big) \R$. 
$(1, \alpha)^T \R$. 
The internal space $H$ is the orthogonal complement of $G$ in $\R^2$.
The CPS consists of the projected points $z+g$ of $z+\Z^2$ (black points) 
where $G$ intersects the adjacent convex set $y+P$ ($g \in \Z^2$). 
The weighted CPS
is obtained by attaching to each point $x = \pi_1(z+g) \in \Lambda$ 
the length of the intersection of $G$ with $g+P$. (The weights
are not shown in the image.) Hence the weight 
function $h$ on $W$ is given by the width of $P$ in direction $G$
(indicated on the left). \label{fig:cps-weighted}}
\end{figure}

The map $h  \colon  H \to \R$ is defined by letting $h(\pi_2(y))$ 
(for $y \in \R^2$) be the length of 
$\big( \begin{smallmatrix} 1\\ \alpha \end{smallmatrix} \big) \R \cap (y+P)$. 
Clearly, $h$ vanishes outside $W$, and
each $P$ fulfilling either the conditions of Theorem \ref{thm:mainpoly} 
or of Theorem \ref{thm:mainconvex} yields a map $h$ that is continuous 
on $H$: the support of $h$ is $W$, and $h(h_1)=h(h_2)=0$. Hence 
$\sum\limits_{x \in \Lambda} h(x^{\star}) \delta_x$ is a weighted CPS. 
(Recall that $x^{\star} = \pi_2(\pi_1^{-1}(x))$ for $x \in \Lambda$.)

Conversely, given a weighted CPS $\Lambda$ with data 
$G = \big( \begin{smallmatrix} 1\\ \alpha \end{smallmatrix} \big) \R,
H=\big( \begin{smallmatrix} 1\\ \alpha \end{smallmatrix} \big)^{\perp},
\Gamma = \Z^2, W=[a,b],\Gamma=\Z^2,h$, 
we can apply the opposite construction to obtain a candidate for 
a BRS with respect to a continuous rotation on the torus. One possible
problem is that the window for $\Lambda$ may be too large to fit into a 
standard fundamental domain of the lattice $\Z^2$. One way to handle 
this is to split the ``big'' CPS into smaller ones.

\begin{lem}
Let $n\in \N$.
A CPS $\Lambda$ with lattice $\Gamma = \Z^2$, $G = \big( \begin{smallmatrix} 1\\ \alpha \end{smallmatrix} \big) \R$, 
$H=G^{\perp}$, and $W=[a,b] \subset H$ is the union of $n^2$ CPS
with lattice translates $\Gamma_{k,\ell} = (k, \ell)^T
+ n\Z^2$ $(0 \le k,\ell \le n-1)$, and the same $G$, $H$, $W$.  
\end{lem}

\begin{proof}
It is enough to notice that $\Gamma=\bigsqcup_{k,\ell =0}^{n-1}\Gamma_{k,\ell}$.
\end{proof}

Hence we assume without loss of generality in the following that 
$W$ fits into the interior of the projection of the fundamental domain 
$[0,1)^2$ of $\Z^2$ along $G$. Otherwise we split the CPS into 
$n^2$ smaller ones as in the lemma above for appropriate large enough $n$. Such a number $n$ exists because the projection of the fundamental domain of $n\Z^2$ is $n$ times bigger than the projection of the fundamental domain of $\Z^2$, and the window $W$ is bounded.

Now we choose a compact set $P \subset [0,1]^2$ such that for 
$z \in W$ the value $h(z)$ equals the length of
$(z+\big( \begin{smallmatrix} 1\\ \alpha \end{smallmatrix} \big) \R) \cap P$. 
(For instance, if $h(z) \ge 0$, then
$P$ can be the region between the graph of $\frac{1}{2} h(z)$ and 
the graph of $-\frac{1}{2} h(z)$.) Now again, the values of $h$ may 
be too large to fit $P$ into $[0,1)^2$. Hence, if needed, we may rescale 
$h$ by some appropriate factor $c'>0$ such that $P$ fits into $[0,1)^2$. 

\begin{lem} \label{lem:wcps-brs}
Let $\omega = \sum\limits_{x \in \Lambda} h(x^{\star}) \delta_x$ and $P$, that depends on $h$, 
be as in the preceding construction. The  weighted CPS $\omega$ is 
bounded distant equivalent to $c \lambda$ for some $c>0$, if and
only if $P$ is a BRS with respect to $\alpha$.
\end{lem}
\begin{proof}
We compare $\Delta_t(P,\alpha)$ with $\sum\limits_{\substack
{x \in \Lambda\\0 \le x \le t}} h(x^{\star}) - 
t \int\limits_{h_1}^{h_2} h(s) ds$.
By construction we have $\lambda(P)=\int\limits_{h_1}^{h_2} h(s) ds$.
Also by construction, $h(x^{\star})$ is the width of the intersection
of the line segment $\{ (s,\alpha s)^T \, | \, \lfloor x \rfloor \le
s \le \lfloor x \rfloor+1 \}$ with a translation of $P$ by an integer vector. So for $t \in \N$ we have
\[
\sum\limits_{\substack{x \in \Lambda\\ 0 \le x \le t}} h(x^{\star}) - 
t \int\limits_{h_1}^{h_2} h(s) ds
= \int_0^t 1_P(s \bmod 1, \alpha s \bmod 1) dt - t \lambda(P) 
\]
Hence the right hand side is uniformly bounded if and only if the left 
hand side is.
\end{proof}

\begin{rem} \label{rem:almostall}
The authors of \cite{GL} give a precise meaning to the ``almost
all'' in Theorems \ref{thm:mainpoly} and \ref{thm:mainconvex}. 
Namely, the results hold for all $\alpha$ whose continued 
fraction expansion $\alpha=[a_0; a_1, a_2,\ldots]$ satisfies
\begin{equation} \label{eq:frac-alpha}
 \sum_{\ell=0}^m \frac{a_{\ell+1}}{q_{\ell}^{1/2}} \sum_{k=1}^{\ell+1}a_k < C ,
\end{equation}
where $C$ is a constant independent of $m$. Here, $(q_\ell)_{\ell \geq 0}$ is 
the sequence of best approximation denominators for $\alpha$.
In particular this implies that the results hold for all 
$\alpha=[a_0; a_1, a_2,\ldots]$ where the $a_i$ are uniformly bounded
by some constant $c$. This follows from the fact that the $q_n$ grow
at least as fast as $\tau^n$ (where $\tau=\frac{\sqrt{5}+1}{2}$). Then
the sum above is less than the convergent sum
\[ \sum_{\ell=0}^{\infty} \frac{c}{\tau^{\ell/2}} (\ell+1) c. \]
Since many $1 \times 1$ CPS in the literature use quadratic irrationals 
for the slope $\alpha$, and quadratic irrationals have periodic continued
fraction expansion, these results apply to most cases of $1 \times 1$ CPS
studied in the literature.
\end{rem}

%%%% Rev 1, comment 3.

\begin{rem}
The proofs of Theorems \ref{thm:mainpoly} and \ref{thm:mainconvex} from \cite{GL} are based on the connection between BRS and $1$-periodic bounded remainder functions. A $1$-periodic function $f:\R\longrightarrow \mathbb{C}$ is called a {\it bounded remainder function} with respect to an irrational number $\alpha$ if the there is a constant $C$ such that 
$$\left|\sum\limits_{k=0}^{N-1}f(k\alpha)-N\int_0^1f(x)dx \right|$$
for all integers $N>0$.

Two results \cite[Props. 2.5 and 2.6]{GL} by Grepstad and Larcher state that a $\Z$-periodization of a positive hat-function (``simplest'' continuous piecewise linear function with compact support) or a $\Z$-periodization of a positive dome-function (a certain twice-differentiable  function inside its compact support, continuous everywhere, with bounded growth/decay at the boundary points of support) are bounded remainder functions. These propositions are the main building blocks for the proofs of Theorems \ref{thm:mainpoly} and \ref{thm:mainconvex}.

In the same way as we have shown how one can transfer the notion of BRS to the notion of weighted CPS, we can transfer bounded remainder functions to weighted CPS and vice versa. The non-weighted CPS $\Lambda$ can be treated as a weighted CPS with non-continuous weight function $h$ being the indicator function of the window $W$. If $\Lambda$ is bounded distance equivalent (as a point set) to a lattice, then the corresponding bounded remainder function will be a $1$-periodic piecewise constant function. Later in Theorem \ref{thm:diraccomb} we will see that weighted $1\times 1$ CPS with many continuous weight functions are bounded distance equivalent to lattices. This is in contrast to the case of non-weighted CPS, where Kesten's theorem \cite{Kes} shows that in the non-weighted case the conditions are more restrictive.

We refer to \cite{Sch} for more discussion on the difference between continuous bounded remainder functions and piecewise constant bounded remainder functions.
\end{rem}

\section{Main results}

Using the results from the last section we can now prove
the following result.

\begin{thm} \label{thm:diraccomb}
Let $\Lambda$ be a $1 \times 1$ CPS with lattice $\Gamma = \Z^2$, 
$G = \big( \begin{smallmatrix} 1\\ \alpha \end{smallmatrix} \big) \R$ and $H=G^{\perp}$, window $W=[a,b] \subset H$,
and let $h$ be continuous on $H$ with support $W$ (i.e., $h$ vanishes outside $W$, 
and $h(x) \ne 0$ for $x$ in the interior of $W$). Furthermore, let $\alpha$ fulfill
the condition \eqref{eq:frac-alpha} in Remark $\ref{rem:almostall}$. 
\begin{enumerate}
\item If $h$ is piecewise linear, or
\item if $h$ is twice differentiable on $W$, and $h''$ is uniformly bounded
on $W$,
\end{enumerate}
then the weighted Dirac comb
$\omega = \sum\limits_{x \in \Lambda} h(x^{\star}) \delta_x$ 
is bounded distance equivalent to $m \lambda$, where 
$m = \int\limits_a^b h(t) dt$.
\end{thm}

\begin{rem}
Though the theorem above is stated for a scaled Lebesgue measure $m\lambda$, it is also true for any measure which is bounded distance equivalent to $m\lambda$. In particular we can use any $t$-periodic measure $\mu$ with $\mu([0,t[)=mt$, or the Dirac comb associated with the lattice $\frac{1}{m}\mathbb{Z}$ of density $m$, or its translates. 

Indeed, let $\mu$ be a $t$-periodic measure with $\mu([0,t[)=m$. Given $a,b$ with $a<b$, let $n$ by the largest integer such that $a+tn\leq b$. Then
$$\mu([a,b])=\mu([a,a+tn[)+\mu([a+tn,b])=n\cdot mt+\mu([a+tn,b]),$$
and the difference $|\mu([a,b])-m\lambda([a,b])|=|\mu([a+tn,b])-m\lambda([a+tn,b])|$ does not exceed $mt$.
\end{rem}

\begin{proof}
Let us first assume that $h$ is twice differentiable on $W$, and $h''$ 
is uniformly bounded on $W$. Choose a compactly supported 
twice differentiable $f$ (we require that $f$ must be twice differentiable in the interior of its support, not at the endpoints), such that the support of $h$ is 
contained in the interior of the support of $f$, and such that there is $c_0>0$
such that the second derivative of $f$ is less than $-c_0$. We will take
$f$ to be the width function of an appropriate big circle $\omega$. Choose 
$c_1>0$ such that the second derivative of $c_1 f-h$ is bounded away
from 0. I.e., there is $c_2<0$ such that for all $t \in W$ holds:
$(c_1 f(t)-h(t))''< c_2 $. Then $c_1 f-h$ is twice differentiable, $c_1 f-h$ has 
negative second derivative less than $c_2<0$, and consequently $c_1f-h$ 
is convex. %At the endpoints of the support of $f$ the function $c_1f-h$ has vertical tangents, but its graph has positive curvature at these points because the curvature will coincide with the curvature of the ellipse $c_1f$ because $h=0$ outside of $W$.

Both $c_1 f$ and $c_1 f -h$ yield convex sets $P$, $P'$ that 
fulfill the conditions of Theorem \ref{thm:mainconvex}:
The convex set $P$ for $c_1 f$ is just an ellipse which is the $c_1$-dilation of the circle $\omega$. Therefore $P$ has positive curvature as any ellipse has positive curvature (this can be checked using the parametrization $x=a\cos \theta, y=b\sin \theta$ and the formulas from Remark \ref{rem:curvature}). As $P'$ 
we might again choose the region between the graphs of
$\pm\frac{1}{2}(c_1f-h)$. The curvature of $P'$ is positive in $W$ because the second derivative of both functions $\pm\frac{1}{2}(c_1f-h)$ is bounded from zero by $\pm\frac{c_2}{2}$ in $W$. Hence the numerator from the formulas of Remark \ref{rem:curvature} can not equal $0$. The curvature of $P'$ is positive in $\supp(f)\setminus W$ because the functions $\pm\frac{1}{2}(c_1f-h)$ coincide with $\pm\frac{c_1f}{2}$ in $\supp(f)\setminus W$, and therefore the curvature of $P'$ is equal to the curvature of $P$ in $\supp(f)\setminus W$, hence non-zero.

Thus both $P$ and $P'$ yield BRS. 
By Lemma \ref{lem:sum-diff-brs} the difference 
$P \setminus P'$ of two BRS $P, P'$ with $P' \subset P$ is 
again a BRS, hence $h$ corresponds to a BRS,
too. By Lemma \ref{lem:wcps-brs} the claim follows.

The case of piecewise linear $h$ is handled analogously. % Rev 2, comment 12
Note that if $h$ is piecewise linear and continuous on $H$, 
then the corresponding polygon $P$ has no edge parallel
to $\big( \begin{smallmatrix} 1\\ \alpha \end{smallmatrix} \big)$. 
\end{proof}

Since Lemma \ref{lem:wcps-brs} and Lemma \ref{lem:sum-diff-brs} 
imply that the sum $\mu_1+\mu_2$ of two measures $\mu_1, \mu_2$
that are bounded distance equivalent with $c_1 \lambda$, respectively 
$c_2 \lambda$, is bounded distance equivalent to $(c_1+c_2) \lambda$, 
the following result is immediate.

\begin{cor}
Any linear combination of Dirac combs as in Theorem $\ref{thm:diraccomb}$
is again bounded distance equivalent to $c \lambda$, for some appropriate
$c>0$. 
\end{cor}

Theorem \ref{thm:diraccomb} holds for almost all $\alpha$, 
more precisely: for all $\alpha$ fulfilling Equation \eqref{eq:frac-alpha}. 
In particular, Theorem \ref{thm:diraccomb} holds for all $\alpha$ 
with bounded values in their continued fraction expansion. 
However, there is no particular example of an algebraic number
of degree larger than two where it is known whether the 
values in its continued fraction expansion are bounded.
Fortunately, many $1 \times 1$ CPS in the literature arise 
from two-letter substitutions \cite{BG}. The slope $\alpha$ for a CPS for 
some two letter substitution is always a quadratic irrational,
compare for instance with Example \ref{ex:fib1}. Since quadratic
irrationals have periodic continued fraction expansions, Theorem 
\ref{thm:diraccomb} holds for all quadratic irrationals $\alpha$.
For a further discussion of the connection between symbolic substitutions
or tile substitutions and (non-weighted) CPS see \cite{HoSo}, or, in the context of BRS,
see \cite{FG} and references therein.

Unfortunately, the most natural way to describe a CPS for
a two-letter substitution is to use a lattice different to $\Z^2$,
namely the one spanned by the vectors $(1,1)^T, (\beta, \beta')^T$,
where $1,\beta$ are the natural tile lengths, and $\beta'$ is the 
algebraic conjugate of $\beta$, see \cite{BG} for details. 
Hence $\beta \beta' = \frac{p}{q} \in \Q$.

The following corollary shows how we can transform Theorem \ref{thm:diraccomb} in order to make it applicable to all weighted CPS with appropriate weight function provided underlying non-weighted CPS arises from a two-letter symbolic substitutions.

\begin{cor}\label{cor:quad}
Let $\beta$ be a quadratic irrational. Let $\Lambda$ be a 
weighted $1 \times 1$ CPS with $G=\R$, 
$\Gamma = \langle (1,1)^T, (\beta, \beta')^T \rangle_{\Z}$,
the window $W=[a,b]$ an interval in $H$ and $h$ as in Theorem 
$\ref{thm:diraccomb}$. Then the Dirac comb
$\omega = \sum\limits_{x \in \Lambda} h(x^{\star}) \delta_x$ 
is bounded distance equivalent to $m \lambda$ where 
$m = \frac{1}{|\beta-\beta'|} \int\limits_a^b h(t) dt$.
\end{cor}
\begin{proof}
The lattice $\Gamma$ can be mapped to the standard integer lattice 
$\Z^2$ by applying some matrix $M$, where $M^{-1}=\big( \begin{smallmatrix}
1 & \beta\\ 1 & \beta' \end{smallmatrix} \big)$. Hence  
$M= \frac{1}{\beta-\beta'} \big( \begin{smallmatrix}
-\beta' & \beta \\ 1 & -1 \end{smallmatrix} \big)$. The slope
$\alpha$ of Theorem \ref{thm:diraccomb} is then 
\[ \alpha = M \big( \begin{smallmatrix} 1 \\ 0 \end{smallmatrix} \big)
= \frac{1}{\beta-\beta'} \big( \begin{smallmatrix} -\beta' \\ 1 
\end{smallmatrix} \big). \]
Hence 
\[ \alpha \R =  \big( \begin{smallmatrix} -\beta' \\ 1 
\end{smallmatrix} \big) \R =  \big( \begin{smallmatrix} \frac{-p}{q} \\ \beta 
\end{smallmatrix} \big) \R. \]
Because of the symmetry of $\Z^2$ the slope $(\frac{-p}{q},\beta)^T$ 
yields the same CPS as the slope $(\frac{p}{q}, \beta)^T$, respectively
the slope $(1,\frac{q}{p} \beta)^T$. Hence the slope $\alpha$ equals 
$\frac{q}{p}\beta$. In particular, $\alpha$ is a quadratic irrational as well.
Furthermore, $M$ preserves the properties 
of $h$. By Theorem \ref{thm:diraccomb} the resulting CPS $\Lambda'$
is bounded distance equivalent to $c' \lambda$ for some appropriate 
$c'$. Since the original CPS is just the image of $\Lambda'$ under some
(regular) linear map, $\Lambda$ is also bounded distance equivalent to
$m \lambda$ for some appropriate $m$. By the density formula for weighted
CPS \eqref{eq:dens-weighted} holds $m = \frac{1}{|\det(M^{-1})|} = 
\frac{1}{|\beta-\beta'|} \int\limits_a^b h(t) dt$.
\end{proof}

\begin{exam}\label{ex:fib2}
The (symbolic) Fibonacci sequence can be generated by applying
the map $\sigma \colon  \, a \mapsto ab$, $b \mapsto a$ repeatedly to 
the letter pair $a|a$: $\sigma(a|a)=ab|ab$, $\sigma^2(a)=aba|aba$, 
$\sigma^4(a)=abaababa|abaababa$, $\sigma^6(a)=abaababaabaababaababa|
abaababaabaababaababa$, $\ldots$. This symbolic sequence can be 
transformed into a Delone set in $\R$ by assigning an interval of length 
$\tau = \frac{\sqrt{5}+1}{2}$ to $a$ and an interval of length 1 to $b$. 
The corresponding Delone set $\Lambda$ then consists of the endpoints of the intervals.
This Delone set can be defined via a CPS, too, and the corresponding CPS is given in \ref{ex:fib1}.

Since $\alpha$ is a quadratic irrational, then we can apply Corollary \ref{cor:quad} to weighted CPS defined with the data of the Fibonacci sequence with an appropriate weight function $h$. In particular, if $h$ is continuously twice differential or continuous piecewise linear and supported by the window of $\Lambda$, then the corresponding weighted CPS is bounded distance equivalent to $m\lambda$ with appropriate $m$.

The original Fibonacci sequence can be treated as a weighted CPS with $h$ being the indicator function of $W$. The corresponding Dirac comb is bounded distance equivalent to $m\lambda$ for some $m$ due to Kesten's theorem \cite{Kes}. However if $h$ is the indicator function of (either) half of $W$ then the resulted weighted CPS corresponds to a Half-Fibonacci sequence and is not bounded distance equivalent to $m\lambda$ for any $m$ due to Kesten's theorem again. Here we would like to refer to \cite{Sch} for more details between continuous and piecewise constant bounded remainder functions.
\end{exam}

\section{A remark on higher dimensions}

Most of the basic objects discussed in this paper can be generalized in higher dimensions. In particular, the definition of weighted CPS will not change if we set direct space to be $d$-dimensional, so $G=\R^d$, and internal space to be $n$-dimensional, so $H=\R^n$.

However, if $d\geq 2$, then the definition of bounded distance equivalent measures will probably be more complicated than in the one-dimensional case. As it can be seen from \cite[Section 1]{Lac}, even for an (unweighted) Dirac comb $\mu$, the condition we need to check in order to see whether $\mu$ is bounded distance equivalent to a Dirac comb corresponding to a lattice (in the same sense as bounded distance equivalence of discrete sets), it is not enough to check the discrepancy of measures on one sequence of growing regions, balls or cubes. Probably, the best definition of bounded distance equivalence for measures will be the definition related to a {\it transportation measure} from \cite{ST}. The transformation from weighted CPS to BRS will work in this case to some extent. For example, a BRS should be defined using $d$-dimensional integrals in that case. However, we don't have any results in this direction.

If $d=1$ but $n>1$, which is the case of one-dimensional CPS with higher dimensional internal space, then all the notions that are defined for objects in the direct space, including bounded distance equivalence for measures, stay the same. However the notions defined in the internal space should be substituted with their higher-dimensional analogs. In particular, the weight function $h$ for a weighted CPS should be defined on an $n$-dimen\-sion\-al region. The transformation from ``weighted CPS bounded distance equivalent to a lattice'' to ``BRS of continuous rotation'' will work in the same way as in the case of $1\times 1$ weighted CPS, but the corresponding BRS now will be in $(n+1)$-dimensional torus $\mathbb{T}^{n+1}=[0,1]^{n+1}$. In this case we are unaware about any results, except the results in \cite{GLev} that can be transformed to unweighted CPS using the approach from \cite{HK}.

\section*{Acknowledgment}
Both authors are grateful to the anonymous referees for several valuable remarks.
DF thanks the Research Centre of Mathematical
Modelling (RCM$^2$) at Bielefeld University for financial support. AG thanks CRC 701 at Bielefeld University for financial support and hospitality.

\end{document}